\documentclass[12pt]{amsart}

\copyrightinfo{2019}{}

\usepackage{latexsym, amsmath, amssymb,amsthm,amscd,amsopn,amsfonts,mathrsfs,tikz}
\usepackage{stmaryrd}
\usepackage{version}
\usepackage{epsfig,graphics,color,graphicx,graphpap,dsfont,eufrak}
\usepackage{amssymb}
\usepackage{tikz}
\usetikzlibrary{snakes}

\ifx\pdfoutput\undefined
  \DeclareGraphicsExtensions{.eps}
\else
  \ifx\pdfoutput\relax
    \DeclareGraphicsExtensions{.eps}
  \else
    \ifnum\pdfoutput>0
      \DeclareGraphicsExtensions{.pdf}
    \else
      \DeclareGraphicsExtensions{.eps}
    \fi
  \fi
\fi

\usepackage{subfigure}

\usepackage{verbatim}

\newcommand{\R}[0]{\mathbb R}

\newcommand{\Ds}[0]{\mathcal D}

\newtheorem{Th}{Theorem}[section]
\newtheorem{Lemma}{Lemma}[section]
\newtheorem{Prop}[Lemma]{Proposition}

\begin{document}

\title[Nowhere-differentiability of the solution map]{Nowhere-differentiability of the solution map of 2D Euler equations on bounded spatial domain}

\author[Hasan Inci and  Y. Charles Li]{Hasan Inci\\
Ko\c{c} \"Universitesi Fen Fak\"ultesi, Rumelifeneri Yolu\\
34450 Sar{\i}yer \.{I}stanbul T\"urkiye\\   \vspace{0.25in} 
        {\it Email: } hinci@ku.edu.tr\\   
         Y. Charles Li\\
        Department of Mathematics, University of Missouri, Columbia, MO 65211, USA \\
{\it Email:} liyan@missouri.edu}

\subjclass{76, 35}

\keywords{Nowhere-differentiability, nowhere locally uniformly continuous, solution map, Euler equations}
\date{}

\begin{abstract}
We consider the incompressible 2D Euler equations on bounded spatial domain $S$, and study the 
solution map on the Sobolev spaces $H^k(S)$ ($k > 2$). Through an elaborate geometric 
construction, we show that for any $T >0$, the time $T$ solution map $u_0 \mapsto u(T)$ is nowhere locally uniformly continuous and nowhere Fr\'echet differentiable.
\end{abstract}

\maketitle

\tableofcontents

\section{Introduction}\label{section_introduction}

The initial value problem for the incompressible 2D Euler equations on bounded spatial domain $S$
is given by
\begin{equation}
u_t + (u \cdot \nabla) u = -\nabla p , \quad 
\operatorname{div} u = 0 , \quad u(0) = u_0 ,
\label{euler}
\end{equation}
under the slip boundary condition on the boundary of $S$, 
where $u:\R \times S \to \R^2$ is the velocity vector field of the flow and $p:\R \times S \to \R$ is the pressure field. Typical bounded spatial domains are the 2D torus (in which case, spatially periodic boundary condition is enforced) and a periodic section of the channel flow (in which case, stream-wise periodic boundary condition is enforced). For convenience of presentation, from now on we 
will use the 2D torus $\mathbb T^2$ to represent the bounded spatial domains.
The initial value problem \eqref{euler} is globally well posed in $H^k_\sigma(\mathbb T^2;\R^2)$ for $k > 2$ \cite{ebin_marsden} \cite{wolibner,holder}. Here we denote by $H^k_\sigma(\mathbb T^2;\R^2)$ the divergence free vector fields on $\mathbb T^2$ of Sobolev class $H^k$. The system \eqref{euler} is invariant under the scaling 
\[
\lambda u(\lambda t), \lambda^2 p(\lambda t)
\]
for $\lambda > 0$. For each $T > 0$, denote by $\Phi_T$ the solution map
\[
 \Phi_T: H^k_\sigma(\mathbb T^2;\R^2) \to H^k_\sigma(\mathbb T^2;\R^2),\quad u_0 \mapsto u(T) 
\]
mapping the initial value $u_0$ to the value of the solution at time $T$. $\Phi_T $ is a continuous map. Our main result is

\begin{Th}\label{th_nonuniform}
The solution map
\[
 \Phi_T: H^k_\sigma(\mathbb T^2;\R^2) \to H^k_\sigma(\mathbb T^2;\R^2)
\]
is nowhere locally uniformly continuous and nowhere Fr\'echet differentiable.
\end{Th}

The physical significance of the theorem can be briefly summarized as follows: In nonlinear 
chaotic dynamics, an important measure is the (maximal) Liapunov exponent 
which is the log of the norm of the derivative of the solution map. A positive Liapunov exponent 
is an indicator of chaotic dynamics. The norm of the derivative of the solution map characterizes 
the maximal rate of the amplification of perturbations. A positive Liapunov exponent implies that 
the maximal rate of the amplification of perturbations is exponential - sensitive dependence on initial data. For the Euler equations of fluids, our theorem states that the derivative of the solution map nowhere exists. The common way 
of such a non-existence is that the norm of the derivative of the solution map is infinite. Thus 
the maximal rate of the amplification of perturbations to Euler equations is infinite - rough dependence on initial data \cite{Li17}. 

The theorem holds when $\mathbb T^2$ is replaced with other bounded domains. 
In \cite{himonas}, a sequence of explicit solutions is constructed to show that the solution map 
is not uniformly continuous on the sequence. In \cite{Li17}, explicit solutions are constructed 
to show that the solution map is not differentiable along these solutions.
Theorem \ref{th_nonuniform} in the case where the spatial domain is the whole space $\R^d$ ($d=2,3$) was proven in \cite{euler}. The bounded spatial domain case is a challenge. In the current paper, we are able to succeed in 2D. The 3D case is still open. 
In contrast to the whole-space case, there are some difficulties in the bounded spatial domain case. In the whole-space case, one has the advantage of dealing with compact-support initial condition of the base-solution, and the interaction of such initial condition with the added ``pulses'' can be eliminated by putting the 
``pulses" far away. Such an arrangement is not possible in the bounded spatial domain case. Let us sketch briefly the strategy of the current proof.
The equations \eqref{euler} have a rich geometric structure. It is well known that one can formulate \eqref{euler} as an ODE in Lagrangian coordinates. More precisely, consider a solution $u$ of \eqref{euler}, and introduce its flow map $\varphi$ as
\[
 \varphi_t = u \circ \varphi,\quad \varphi(0)=\mbox{id}
\]
where $\mbox{id}$ is the identity map. It turns out that \eqref{euler} is equivalent to a second order ODE
\[
 \varphi_{tt}=F(\varphi,\varphi_t) .
\]
In particular, we have a smooth dependence in Lagrangian coordinates, i.e.
\begin{equation}
 u_0 \mapsto \varphi(T)
\label{smdep}
\end{equation}
is smooth. This smooth dependence is the first ingredient. The second ingredient is the Cauchy theorem on vorticity which demonstrates the vorticity's property of being ``frozen'' into the flow \cite{MP94}. In 2D case, it has the simple form
\[
 \omega(T)=\omega_0 \circ \varphi(T)^{-1}
\]
where $\omega(T)$ and $\omega_0$ are the vorticities at times $T$ and $0$ respectively. In order to establish a nonuniform-continuity, we construct $\omega_0$ and $\tilde \omega_0$ which differ slightly but produce a considerable difference for the corresponding $\omega(T)$ and $\tilde \omega(T)$. To achieve that, we need some control over $\varphi(T)$ which can be obtained through the smooth dependence in Lagrangian coordinates \eqref{smdep}.

\section{A geometric Lagrangian formulation of Euler equations}\label{section_geometry}

The concepts of this section were already used in the first local well posedness results for \eqref{euler}, see \cite{lichtenstein,gunter}. They became very popular through \cite{arnold} and subsequently \cite{ebin_marsden}.
Assume that we have a solution $u=(u_1,\ldots,u_d)$ to the Euler equation
\begin{equation}
 u_t + (u \cdot \nabla) u = -\nabla p ,
\label{geuler}
\end{equation}
where $d=2, 3$.
Taking the divergence, we end up with
\[
 -\Delta p = \sum_{i,j=1}^d \partial_i u_j \partial_j u_i .
\]
Solving for $-\nabla p$ gives
\[
 -\nabla p = \Delta^{-1} \nabla \sum_{i,j=1}^d \partial_i u_j \partial_j u_i .
\]
Since $\Delta^{-1}$ is defined on functions with vanishing mean, this makes perfectly sense. Taking the $t$ derivative of $\varphi_t=u \circ \varphi$ gives
\[
 \varphi_{tt}=(u_t + (u \cdot \nabla) u) \circ \varphi = -\nabla p \circ \varphi .
\]
Or  replacing $-\nabla p$, we get
\begin{eqnarray}
 \varphi_{tt} &=& \left(\Delta^{-1} \nabla \sum_{i,j=1}^d \partial_i ((\varphi_t)_j \circ \varphi^{-1}) \cdot \partial_j ((\varphi_t)_i \circ \varphi^{-1})\right) \circ \varphi \nonumber \\
& =:& F(\varphi,\varphi_t) . \label{analytic_ode}
\end{eqnarray}
The right functional space for $\varphi$ is $\Ds^k(\mathbb T^d)$, the group of orientation preserving diffeomorphisms of Sobolev class $H^k$. It turns out that $F(\varphi,\varphi_t)$ is analytic on these spaces -- for details see \cite{ebin_marsden,composition,lagrangian}. By solving \eqref{analytic_ode} with initial values $\varphi(0)=\mbox{id}, \varphi_t(0)=u_0$ up to time $T=1$, we get an analytic exponential map
\[
 \exp:U \subset H^k_\sigma(\mathbb T^d;\R^d) \to \Ds^k(\mathbb T^d),\quad u_0 \mapsto \varphi(1)
\]
which gives a complete description of the solutions to \eqref{geuler}. For more details on exponential maps, see \cite{lang}.

\section{Nowhere-uniform continuity of the solution map}

The vorticity of $u(t)$ in the 2D case is the scalar
\[
 \omega(t):=\partial_1 u_2(t) - \partial_1 u_1(t).
\]
By the Biot Savart law, we have for divergence free $u$
\[
 ||\nabla u||_{H^{k-1}} \leq C ||\omega||_{H^{k-1}}
\]
for some $C > 0$. Moreover, the vorticity is ``frozen'' into the fluid flow in the sense that
\begin{equation}\label{vorticity}
 \omega(t) \circ \varphi(t)=\omega_0,\quad \forall t
\end{equation}
where $\varphi$ is the flow map of $u$ ($\varphi_t=u \circ \varphi, \varphi(0)=\mbox{id}$) and $\omega_0$ is the initial vorticity \cite{MP94}. Because of the scaling $\lambda u(\lambda t)$, it will be enough to establish Theorem \ref{th_nonuniform} for the case $T=1$ to get the same conclusion for the full range $T > 0$. More precisely, if we denote by $\Phi$ the $T=1$ solution map, then
\[
 \Phi_T(u_0)=\frac{1}{T} \Phi(T \cdot u_0).
\]
\begin{Prop}\label{prop_nonuniform}
Let $\Phi=\left. \Phi_T \right|_{T=1}$ be the time-1 solution map. Then
\[
 \Phi: H^k_\sigma(\mathbb T^2;\R^2) \to H^k_\sigma(\mathbb T^2;\R^2),
\]
is nowhere locally uniformly continuous.
\end{Prop} 

Before proving this proposition, we prove the following technical lemma which tells us that the exponential map is not locally constant.

\begin{Lemma}\label{lemma_dense}
There is a dense subset $S \subseteq H^k_\sigma(\mathbb T^2;\R^2)$ with $S \subseteq C^\infty$ such that
\[
 d_{u_\bullet} \exp \neq 0,\quad \forall u_\bullet \in S
\]
where
\[
 d_{u_\bullet} \exp: H^k_\sigma(\mathbb T^2;\R^2) \to H^k(\mathbb T^2;\R^2)
\]
is the differential of $exp:H^k_\sigma \to \Ds^k$ at $u_\bullet$.
\end{Lemma} 

\begin{proof}
Take an arbitrary $u_\bullet \in C^\infty$. Take $w \in H^k_\sigma(\mathbb T^2;\R^2)$ and $x^\ast \in \mathbb T^2$ with $w(x^\ast) \neq 0$. Consider the analytic curve
\[
 \gamma:[0,1] \to \R^2,\quad t \mapsto (d_{t u_\bullet}\exp(w))(x^\ast)
\] 
As $d_0 \exp = \mbox{id}$ (see \cite{lang}), we have $\gamma(0)=w(x^\ast) \neq 0$. Because of analyticity, we get infinitely many $t_n \uparrow 1$ with $(d_{t_n u_\bullet}\exp(w^\ast))(x^\ast) \neq 0$. Thus we can put all these $t_n u_\bullet$ into $S$. This construction gives a dense subset $S$ consisting of $C^\infty$ vector-fields.
\end{proof}

\begin{proof}[Proof of Proposition \ref{prop_nonuniform}]
Let $u_\bullet \in S$ be as in Lemma \ref{lemma_dense} with a corresponding $x^\ast \in \mathbb T^2$ and $w_\ast \in H^k_\sigma(\mathbb T^2;\R^2)$ such that
\begin{equation}\label{w_ast}
 m:=|(d_{u_\bullet}\exp(w_\ast))(x^\ast)| \neq 0 .
\end{equation}
In the following, we will determine a $R_\ast > 0$ and prove that
\[
 \left. \Phi \right|_{B_R(u_\bullet)}:B_R(u_\bullet) \subseteq H^k_\sigma(\mathbb T^2;\R^2) \to H^k_\sigma(\mathbb T^2;\R^2)
\]
is not uniformly continuous for any $0 < R < R_\ast$. Here $B_R(u_\bullet)$ denotes the ball of radius $R$ in $H^k_\sigma(\mathbb T^2;\R^2)$ around $u_\bullet$. As $S$ is dense in $H^k_\sigma(\mathbb T^2;\R^2)$, this clearly suffices.
First we choose $R_1 > 0$ small enough and $C_1 > 0$ with
\begin{equation}\label{below_above}
 \frac{1}{C_1} ||f||_{H^{k-1}} \leq ||f \circ \varphi^{-1}||_{H^{k-1}} \leq C_1 ||f||_{H^{k-1}}
\end{equation}
for all $f \in H^{k-1}(\mathbb T^2;\R^2)$ and for all $\varphi \in \exp(B_{R_1}(u_\bullet))$. That this is possible follows from the continuity properties of the composition -- see \cite{composition}. Using the Sobolev embedding theorem, we choose $0 < R_2 < R_1$ and $C_2 > 0$ such that
\begin{equation}\label{lipschitz}
 |\varphi(x)-\varphi(y)| \leq C_2 |x-y|
\end{equation}
for all $x,y \in \mathbb T^2$ and $\varphi \in \exp(B_{R_2}(u_\bullet))$. To make estimates around $\exp(u_\bullet)$, we use the Taylor expansion
\[
 \exp(u_\bullet+h)=\exp(u_\bullet)+d_{u_\bullet}\exp(h) + \int_0^1 (1-s) d_{u_0+sh}^2\exp(h,h) \;ds .
\]
To estimate the second derivatives in this expansion, we choose $0 < R_3 < R_2$ such that
\begin{equation}
  ||d_{\tilde u}^2 \exp(h_1,h_2)||_{H^k} \leq C_3 ||h_1||_{H^k} ||h_2||_{H^k} 
\label{SDE1}
\end{equation}
and
\begin{eqnarray}
 && ||d_{\tilde u_1}^2 \exp(h_1,h_2)-d^2_{\tilde u_2}\exp(h_1,h_2)||_{H^k} \nonumber \\ & \leq& C_3 ||\tilde u_1 - \tilde u_2||_{H^k} ||h_1||_{H^k} ||h_2||_{H^k} 
\label{SDE2}
\end{eqnarray}
for some $C_3 > 0$ and for all $\tilde u, \tilde u_1, \tilde u_2 \in B_{R_\ast}(u_\bullet)$ and all $h_1,h_2 \in H^k(\mathbb T^2;\R^2)$. Due to the smoothness of $\exp$, this is possible. Now let us
fix $C > 0$ in the Sobolev imbedding
\[
 |f(x)| \leq C ||f||_{H^k}, \quad \forall x \in \mathbb T^2
\]
for all $f \in H^k(\mathbb T^2;\R^2)$. Then we choose $0 < R_\ast < R_3$ in such a way that
\[
 ||\varphi^{-1} - \varphi_\bullet^{-1}||_{H^k} < 1 
\]
for all $\varphi \in \exp(B_{R_\ast}(u_\bullet))$, where $\varphi_\bullet=\exp(u_\bullet)$. Making $R_\ast$ smaller if necessary, we can require 
\begin{equation}\label{r_ast}
 (C C_3 R_\ast^2/4 + C C_3 R_\ast) \cdot ||w_\ast||_{H^k} < \frac{m}{4}  
\end{equation}
Finally we fix a $R$ ($0 < R < R_\ast$). Our goal is to construct two sequences of initial values
\[
 (u_0^{(n)})_{n \geq 1}, (\tilde u_0^{(n)})_{n \geq 1} \subseteq B_R(u_\bullet)
\]
such that
\[
 \lim_{n \to \infty} ||u_0^{(n)} - \tilde u_0^{(n)}||_{H^k} = 0,
\]
but
\[
 \limsup_{n \to \infty} ||\Phi(u_0^{(n)}) - \Phi(\tilde u_0^{(n)})||_{H^k} > 0
\]
which would imply that $\Phi$ is not uniformly continuous on $B_R(u_\bullet)$. Denoting by $\omega^{(n)}, \ \tilde \omega^{(n)}$ the vorticities of $\Phi(u_0^{(n)})$, $\Phi(\tilde u_0^{(n)})$ 
respectively, we have obviously
\[
 ||\omega^{(n)}-\tilde \omega^{(n)}||_{H^{k-1}} \leq \tilde C ||\Phi(u_0^{(n)}) - \Phi(\tilde u_0^{(n)})||_{H^k}
\]
for some $\tilde C > 0$. Therefore, it will be enough to establish
\[
\limsup_{n \to \infty} ||\omega^{(n)}-\tilde \omega^{(n)}||_{H^{k-1}} > 0 .
\]
Let us now construct these sequences explicitly. With $w_\ast$ and $x^\ast$ from \eqref{w_ast}, we choose for $n \geq 1$
\begin{equation}
 u_0^{(n)}=u_\bullet + v_n,\quad \tilde u_0^{(n)} = u_\bullet + v_n + \frac{1}{n} w_\ast
\label{IVS}
\end{equation}
where we pick a $v_n \in H^k_\sigma(\mathbb T^2;\R^2)$ with $||v_n||_{H^k}=R/2$ and
\[
 \operatorname{supp} v_n \subseteq B_{r_n}(x^\ast) \subseteq \mathbb T^2,\quad r_n=\frac{m}{8n C_2}
\]
where $\operatorname{supp}$ denotes the support, $B_{r_n}(x^\ast)$ is the ball in $\mathbb T^2$ of radius $r_n$ with center $x^\ast \in \mathbb T^2$, and $C_2$ is the Lipschitz constant from \eqref{lipschitz}. For some large $N$, we have that the initial values \eqref{IVS} lie in $B_R(u_\bullet)$ for $n \geq N$. Furthermore by construction
\[
 \lim_{n \to \infty} ||u_0^{(n)}-\tilde u_0^{(n)}||_{H^k} = \lim_{n \to \infty} ||\frac{1}{n} w_\ast||_{H^k} = 0 .
\]
For $n \geq N$, we introduce
\[
 \varphi^{(n)}=\exp(u_0^{(n)}),\quad \tilde \varphi^{(n)}=\exp(\tilde u_0^{(n)}) .
\]
We then have by \eqref{vorticity}
\[
 \omega^{(n)}=\omega_0^{(n)} \circ (\varphi^{(n)})^{-1},\quad \tilde \omega^{(n)}=\tilde \omega_0^{(n)} \circ (\tilde \varphi^{(n)})^{-1}
\]
where $\omega_0^{(n)}, \  \tilde \omega_0^{(n)}$ are the vorticities of $u_0^{(n)}$, $\tilde u_0^{(n)}$ respectively. So we have to estimate
\begin{equation}
 \limsup_{n \to \infty} ||\omega_0^{(n)} \circ (\varphi^{(n)})^{-1}-\tilde \omega_0^{(n)} \circ (\tilde \varphi^{(n)})^{-1}||_{H^{k-1}} .
\label{EST1}
\end{equation}
By construction, the vorticities decompose at $t=0$ to
\begin{equation}
 \omega_0^{(n)}=\omega_\bullet + \omega_n,\quad \tilde \omega_0^{(n)}=\omega_\bullet + \omega_n + \frac{1}{n} \omega_\ast .
\label{EST2}
\end{equation}
Hence we have to estimate
\begin{equation}
 \limsup_{n \to \infty} ||(\omega_\bullet + \omega_n) \circ (\varphi^{(n)})^{-1}-(\omega_\bullet + \omega_n + \frac{1}{n} \omega_\ast) \circ (\tilde \varphi^{(n)})^{-1}||_{H^{k-1}} .
\label{EST3}
\end{equation}
Clearly,
\[
(\omega_\bullet + \omega_n) \circ (\varphi^{(n)})^{-1} = \omega_\bullet  \circ (\varphi^{(n)})^{-1} + \omega_n \circ (\varphi^{(n)})^{-1},
\]
and 
\[
(\omega_\bullet + \omega_n + \frac{1}{n} \omega_\ast) \circ (\tilde \varphi^{(n)})^{-1} = 
\omega_\bullet \circ (\tilde \varphi^{(n)})^{-1} + \omega_n \circ (\tilde \varphi^{(n)})^{-1} +
\frac{1}{n} \omega_\ast \circ (\tilde \varphi^{(n)})^{-1} .
\]
We have
\begin{eqnarray}
& & ||(\omega_\bullet + \omega_n) \circ (\varphi^{(n)})^{-1}-(\omega_\bullet + \omega_n + \frac{1}{n} \omega_\ast) \circ (\tilde \varphi^{(n)})^{-1}||_{H^{k-1}} \nonumber \\
& & \geq ||\omega_n \circ (\varphi^{(n)})^{-1}-\omega_n \circ (\tilde \varphi^{(n)})^{-1}||_{H^{k-1}} \nonumber \\
& & - ||\omega_\bullet \circ (\varphi^{(n)})^{-1}-\omega_\bullet \circ (\tilde \varphi^{(n)})^{-1}||_{H^{k-1}} - ||  \frac{1}{n} \omega_\ast \circ (\tilde \varphi^{(n)})^{-1}||_{H^{k-1}} .
\label{EST4}
\end{eqnarray}
First we estimate
\[
 || \omega_\bullet  \circ (\varphi^{(n)})^{-1}-\omega_\bullet \circ (\tilde \varphi^{(n)})^{-1}
||_{H^{k-1}}.
\]
This estimate turns out to be the most challenging in 3D due to the fluid particle deformation factor in front the vorticity, and is still elusive. In the 2D case, 
we can estimate this (see \cite{composition}) by
\[
 ||\omega_\bullet  \circ (\varphi^{(n)})^{-1}-\omega_\bullet \circ (\tilde \varphi^{(n)})^{-1}||_{H^{k-1}} \leq \tilde K ||\omega_\bullet||_{H^k} ||(\varphi^{(n)})^{-1}-(\tilde \varphi^{(n)})^{-1}||_{H^{k-1}} .
\]
As $\omega_\bullet$ is fixed and smooth, its $H^k$ norm is bounded. By the Sobolev imbedding we know that
\[
 \varphi^{(n)}-\tilde \varphi^{(n)} \to 0 
\]
in $C^1$, and by the choice of $R_\ast$ we know that the $C^1$ norms of their inverses are bounded, thus
\[
 (\varphi^{(n)})^{-1}-(\tilde \varphi^{(n)})^{-1} \to 0
\] 
uniformly and therefore also in $L^2$. Since the inverses are bounded in $H^k$, we get, by interpolation, convergence to 0 in $H^{k-1}$. For the $\omega_\ast$ term, we get by \eqref{below_above} that
\[
 ||\frac{1}{n} \omega_\ast \circ (\tilde \varphi^{(n)})^{-1}||_{H^{k-1}} \leq \frac{C_1}{n} ||\omega_\ast||_{H^{k-1}} \to 0
\]
as $n \to \infty$. Thus from \eqref{EST1}-\eqref{EST4}, we arrive at
\begin{align*}
 &\limsup_{n \to \infty} ||\omega_0^{(n)} \circ (\varphi^{(n)})^{-1}-\tilde \omega_0^{(n)} \circ (\tilde \varphi^{(n)})^{-1}||_{H^{k-1}}\\
 &=
 \limsup_{n \to \infty} ||\omega_n \circ (\varphi^{(n)})^{-1}-\omega_n \circ (\tilde \varphi^{(n)})^{-1}||_{H^{k-1}}.
\end{align*}
We claim that the supports of the latter terms are disjoint. In order to prove this, we will estimate the ``distance'' between $\varphi^{(n)}$ and $\tilde \varphi^{(n)}$. Using the Taylor expansion we have
\begin{align*}
 \tilde \varphi^{(n)}& = \exp(u_\bullet + v_n + \frac{1}{n}w_\ast) = \varphi_\bullet + d_{u_\bullet}\exp(v_n+\frac{1}{n}w_\ast)  \\
 &+\int_0^1 (1-s) d^2_{u_\bullet+s(v_n+\frac{1}{n}w_\ast)}\exp(v_n+\frac{1}{n}w_\ast,v_n+\frac{1}{n}w_\ast) \;ds 
\end{align*}
and
\begin{align*}
 \varphi^{(n)}& = \exp(u_\bullet + v_n) \\ &= \varphi_\bullet + d_{u_\bullet}\exp(v_n) +\int_0^1 (1-s) d^2_{u_\bullet+s v_n}\exp(v_n,v_n) \;ds .
\end{align*}
We thus have
\[
 \tilde \varphi^{(n)}-\varphi^{(n)}=\frac{1}{n} d_{u_\bullet}\exp(w_\ast) + I^{(n)}_1 + I^{(n)}_2 + I^{(n)}_3
\]
where 
\[
 I^{(n)}_1=\int_0^1 (1-s) \left(d^2_{u_\bullet+s(v_n+\frac{1}{n}w_\ast)}\exp(v_n,v_n)-d^2_{u_\bullet+s v_n}\exp(v_n,v_n)\right) \;ds
\]
and
\[
 I^{(n)}_2=2 \int_0^1 (1-s) d^2_{u_\bullet+s(v_n+\frac{1}{n}w_\ast)}\exp(v_n,\frac{1}{n}w_\ast) \;ds
\]
and
\[
 I^{(n)}_3=\int_0^1 (1-s) d^2_{u_\bullet+s(v_n+\frac{1}{n}w_\ast)}\exp(\frac{1}{n}w_\ast,\frac{1}{n}w_\ast) \;ds .
\]
Using the estimates \eqref{SDE1}-\eqref{SDE2} for the second derivatives, we have
\begin{eqnarray*}
&& ||I^{(n)}_1||_{H^k} \leq \frac{C_3 R^2}{4n} ||w_\ast||_{H^k},\quad ||I^{(n)}_2||_{H^k} \leq \frac{C_3 R}{n} ||w_\ast||_{H^k},
\\ && ||I^{(n)}_3||_{H^k} \leq \frac{C_3}{n^2} ||w_\ast||^2_{H^k} .
\end{eqnarray*}
Hence using the Sobolev imbedding and the choice of $R_\ast$ in \eqref{r_ast}, we have
\begin{align*}
 &|I^{(n)}_1(x^\ast)|+|I^{(n)}_2(x^\ast)|+|I^{(n)}_3(x^\ast)| \\
 &\leq \frac{C C_3 R^2}{4n} ||w_\ast||_{H^k} + \frac{C C_3 R}{n} ||w_\ast||_{H^k} + \frac{C C_3}{n^2} ||w_\ast||^2_{H^k} < \frac{m}{2n}
\end{align*}
for $n \geq N'$ (for some large $N'$), where $m$ is the one from \eqref{w_ast}. Using the triangle inequality, we get
\[
 |\tilde \varphi^{(n)}(x^\ast)-\varphi^{(n)}(x^\ast)| > \frac{1}{n} |d_{u_\bullet} \exp(w_\ast)| - \frac{m}{2n} = \frac{m}{2n} .
\]
Since the supports satisfy
\[
 \operatorname{supp} \omega_n \circ (\varphi^{(n)})^{-1} \subseteq B_{C_2 r_n}(\varphi^{(n)}(x^\ast)) = B_{m/8n}(\varphi^{(n)}(x^\ast))
\]
and
\[
  \operatorname{supp} \omega_n \circ (\tilde \varphi^{(n)})^{-1} \subseteq B_{C_2 r_n}(\tilde \varphi^{(n)}(x^\ast)) = B_{m/8n}(\tilde \varphi^{(n)}(x^\ast)),
\]
we see that their supports are disjoint. We thus can ``separate'' the norms
\begin{align*}
 &\limsup_{n \to \infty} ||\omega_n \circ (\varphi^{(n)})^{-1}-\omega_n \circ (\tilde \varphi^{(n)})^{-1}||^2_{H^{k-1}} \\ 
 & =\limsup_{n \to \infty} (|| \omega_n \circ (\varphi^{(n)})^{-1}||^2_{H^{k-1}} + ||\omega_n \circ (\tilde \varphi^{(n)})^{-1}||^2_{H^{k-1}})\\
 &\geq \limsup_{n \to \infty} \frac{2}{C_1^2} ||\omega_n||^2_{H^{k-1}}
\end{align*}
where we used \eqref{below_above} in the last step. Now note that $||v_n||_{H^k}=R/2$ is fixed whereas its support goes to zero. In particular, we have $||v_n||_{L^2} \to 0$ because
\[
 \int_{\mathbb T^2} |v_n(x)|^2 \leq C^2 ||v_n||^2_{H^k} \cdot \operatorname{vol}(B_{r_n}(x^\ast)) \to 0
\]
as $n \to \infty$. Hence for $n \to \infty$, we have $||v_n||_{H^k} \sim ||\nabla v_n||_{H^{k-1}}$ or more precisely
\[
\limsup_{n \to \infty} ||\nabla v_n||_{H^{k-1}} \geq \hat C_1 ||v_n||_{H^k}
\] 
for some $\hat C_1 > 0$. By the Biot Savart Law
\[
 \limsup_{n \to \infty} ||\omega_n||_{H^{k-1}} \geq \hat C_2  \limsup_{n \to \infty} ||\nabla v_n||_{H^{k-1}} \geq \hat C_1 \hat C_2 R ,
\]
for some $\hat C_2 >0$. Combining everything, we end up with
\begin{equation}\label{linear_estimate}
\limsup_{n \to \infty} ||\Phi(u^{(n)}_0)-\Phi(\tilde u^{(n)}_0)||_{H^k} \geq C_\ast R
\end{equation}
for some $C_\ast > 0$. Note that $C_\ast$ is independent of $R$ for $0 < R < R_\ast$. The proof of the
Proposition is complete.
\end{proof}

\section{Nowhere-differentiability of the solution map}\label{section_diff}

Now we prove that the time $T=1$ solution map 
\[
 \Phi: H^k_\sigma(\mathbb T^2;\R^2) \to H^k_\sigma(\mathbb T^2;\R^2)
\]
is nowhere Fr\'echet differentiable.

\begin{Prop}\label{prop_diff}
The map $\Phi$ is nowhere Fr\'echet differentiable.
\end{Prop}

\begin{proof}
The proof is based on estimate \eqref{linear_estimate}. In the following, we will see that differentiability prevents such an estimate. Take $u_0 \in H^k_\sigma(\mathbb T^2;\R^2)$ and a ball $B \subseteq H^k_\sigma(\mathbb T^2;\R^2)$ around $u_0$ with an estimate as in \eqref{linear_estimate}. To be precise, take $u_\bullet \in S$ near $u_0$ and determine $R_\ast$ and $C_\ast$. A careful examination shows that the choice of $R_\ast$ can be made locally uniformly. Thus there will be a ball $B_{R_\ast}(u_\bullet)$ covering $u_0$. Now assume that $\Phi$ is Fr\'echet differentiable at $u_0$, i.e. for $\tilde u_0$ in a neighborhood of $u_0$, we have 
\[
 \Phi(\tilde u_0)=\Phi(u_0)+d_{u_0}\Phi(\tilde u_0-u_0) + r(\tilde u_0)
\]
with $||r(\tilde u_0)||_{H^k} \leq \frac{C_\ast}{4} ||\tilde u_0 - u_0||_{H^k}$ for $||\tilde u_0 - u_0||_{H^k} \leq \delta$ for some $\delta > 0$ small enough. As we have seen above, we can construct two sequences
\[
 (u_0^{(n)})_{n \geq 1}, (\tilde u_0^{(n)})_{n \geq 1} \subseteq B_\delta(u_0) 
\]
with $||u_0^{(n)}-\tilde u_0^{(n)}||_{H^k} \to 0$ for $n \to \infty$ and
\[
 \limsup_{n \to \infty} ||\Phi(u_0^{(n)})-\Phi(\tilde u_0^{(n)})||_{H^k} \geq C_\ast \delta .
\]
Applying differentiability gives
\[
 \Phi(u_0^{(n)})-\Phi(\tilde u_0^{(n)})=d_{u_0}\Phi(u_0^{(n)}-\tilde u_0^{(n)}) + r(u_0^{(n)})-r(\tilde u_0^{(n)})
\]
which gives the contradiction
\begin{eqnarray*}
&& \limsup_{n \to \infty} ||\Phi(u_0^{(n)})-\Phi(\tilde u_0^{(n)})||_{H^k} \\ && \leq \limsup_{n \to \infty} \left(||r(u_0^{(n)})||_{H^k} + ||r(\tilde u_0^{(n)})||_{H^k}\right) \leq \frac{C_\ast}{2} \delta .
\end{eqnarray*}
Hence $\Phi$ cannot be differentiable at $u_0$. The proof is complete.
\end{proof}
By now, the proof of the main theorem is complete. \\ 

{\bf Acknowledgement}: We would like to thank Professors Dong Li and Jiahong Wu for helpful communication.

\bibliographystyle{plain}

\begin{thebibliography}{999}





\bibitem{arnold} V. Arnold: {\em Sur la g{\'e}ometrie differentielle des
groupes de Lie
de dimension infinie et ses applications {\`a} l'hydrodynamique des
fluids parfaits},
Ann. Inst. Fourier, 16, 1 (1966), 319-361.








\bibitem{ebin_marsden} D. Ebin, J. Marsden: {\em Groups of diffeomorphisms and
the motion of an incompressible fluid}, Ann. Math., 92 (1970), 102-163.




\bibitem{gunter} N. M. Gunter: {\em On the motion of a fluid contained in a given moving vessel}, Izvestia AN USSR, Sect. Phys. Math. 1323-1328, 1503-1532 (1926); 621-656, 1139-1162 (1927); 9-30 (1928).

\bibitem{himonas} A. A. Himonas, G. Misiolek: {\em Non-uniform dependence on initial data of solutions to the Euler equations of hydrodynamics}, Comm. Math. Phys. 296, 1 (2010), 285-301.

\bibitem{holder} E. H\"older: {\em \"Uber die unbeschr\"ankte Fortsetzbarkeit einer stetigen ebenen Bewegung in einer unbegrenzten inkompressiblen Fl\"ussigkeit}, Math. Z., 37 (1933), 727-738.

\bibitem{composition} H. Inci, T. Kappeler, P. Topalov: {\em On the regularity of the composition of diffeomorphisms}, Mem. Amer. Math. Soc. 226 (2013), no. 1062.


\bibitem{euler} H. Inci: {\em On the regularity of the solution map of the incompressible Euler equation}, Dyn. Partial Differ. Equ. 12, 2 (2015), 97-113.


\bibitem{lagrangian} H. Inci: {\em On a Lagrangian formulation of the incompressible Euler equation}, J. Partial Differ. Equ. 29, 4 (2016), 320-359.






\bibitem{lang} S. Lang: {\em Differential and Riemannian Manifolds}, 3rd edition, Springer-Verlag, New York, 1995. 





\bibitem{Li17} Y. Li: {\em Rough dependence upon initial data exemplified by explicit solutions and 
the effect of viscosity}, Nonlinearity 30 (2017), 1097-1108.

\bibitem{lichtenstein} L. Lichtenstein: {\em \"Uber einige Existenzprobleme der Hydrodynamik homogener unzusammendr\"uckbarer, reibungsloser Fl\"ussigkeiten und die Helmholtzschen Wirbels\"atze}, Math. Z. 23, 89-154 (1925); 26, 196-323 (1927); 28, 387-415 (1928); 32, 608-640 (1930).


\bibitem{MP94} C. Marchioro, M. Pulvirenti: {\em Mathematical theory if incompressible nonviscous fluids}, Appl. Math. Sci. 96, Springer, 1994. 





\bibitem{wolibner} W. Wolibner: {\em Un th\'eor\`eme sur l'existence du mouvement plan d'un fluide parfait, homog\`ene, incompressible, pendant un temps infiniment long}, M. Z. 37 (1933), 698-726.

\end{thebibliography}

\end{document}